\documentclass [12pt]{article}
 \usepackage{amssymb}
 \usepackage{amsmath,amsthm}
 \usepackage{mathrsfs}
 \usepackage{color}
 \usepackage{graphicx}

\normalsize

\newtheorem{theorem}{Theorem}[section]
\newtheorem{lemma}[theorem]{Lemma}

\newtheorem{corollary}[theorem]{Corollary}

\newtheorem{remark}[theorem]{Remark}

\numberwithin{equation}{section}

\def\({\bigl(}
\def\){\bigr)}

\begin{document}
\title
{On the number of linear multipartite hypergraphs with given size
}%\thanks{The work was partially supported by NSFC(No.11871377).}

\author
{{Fang Tian}\\
{\small Department of Applied Mathematics}\\
{\small Shanghai University of Finance and Economics, Shanghai, 200433, China}\\
{\small\tt tianf@mail.shufe.edu.cn}
}

\date{}
 \maketitle
\begin{abstract}
For any given integer $r\geqslant 3$, let  $k=k(n)$ be an integer
with $r\leqslant k\leqslant n$.  A hypergraph is $r$-uniform
 if each edge is a set of $r$ vertices, and is said to be linear if
 two edges intersect in at most one vertex. Let $A_1,\ldots,A_k$ be a
 given $k$-partition of $[n]$ with $|A_i|=n_i\geqslant 1$. An
 $r$-uniform hypergraph $H$ is called {\it $k$-partite} if each edge
 $e$ satisfies $|e\cap A_i|\leqslant 1$ for  $1\leqslant i\leqslant k$.
 In this paper, the number of linear $k$-partite $r$-uniform hypergraphs
 on $n\to\infty$ vertices is determined asymptotically when the number of
 edges is $m(n)=o(n^{\frac{4}{3}})$. For  $k=n$, it is the  number of
 linear $r$-uniform hypergraphs on vertex set $[n]$ with $m=o(n^{ \frac{4}{3}})$ edges.
\end{abstract}

\vskip0.4cm \hskip 10pt {\bf Keywords:}\quad asymptotic enumeration,
linear  hypergraph, multipartite hypergraph, switching method.

\hskip 10pt {\bf Mathematics Subject Classifications:}\ 05A16

\date{}
 \maketitle

\section{Introduction}\label{s:1}
%\bdm{All my changes are in red.  Look in the latex file for $\backslash$red and
%$\backslash$bdm to find them all.}

For any given integer $r\geqslant 3$, a hypergraph ${H}$ on vertex set $[n]$ is an
\textit{$r$-uniform hypergraph} (\textit{$r$-graph} for short) if each edge is a set
of  $r$ vertices, and is said to be a \textit{linear hypergraph}  if two edges
intersect in at most one vertex. Little is known about the number of
 distinct linear hypergraphs. An asymptotic enumeration formula for
 the logarithm of the number of linear hypergraphs
 on $n\to\infty$ vertices is obtained by Grable and Phelps~\cite{grable96}. They also obtained
 the logarithm of the number of  partial Steiner $(n,r,\ell)$-systems with $2\leqslant\ell\leqslant r-1$,
 where a partial Steiner $(n,r,\ell)$-system is an $r$-graph $H$ such that
 every subset of size $\ell$ lies in at most one edge of $H$; the $(n,r,2)$-systems
 are linear hypergraphs. Asratian and Kuzjurin~\cite{asas00} gave another proof.
 Blinovsky and Greenhill~\cite{vlaejoc,valelec} used the switching method
 to obtain the asymptotic number of sparse uniform and linear
uniform hypergraphs with given order and degree sequence.
Balogh and Li~\cite{balgoh17} obtained an upper bound on the
 number of  linear uniform hypergraphs with given order and girth.

It is interesting to consider the enumeration of  linear hypergraphs with
given size. Let $\mathcal{H}_r(n,m)$ denote the set of $r$-graphs on
 vertex set $[n]$ with $m$  edges. Let   $\mathcal{L}_r(n,m)$ denote
 the set of  linear hypergraphs in $\mathcal{H}_r(n,m)$. The previous
 works most relevant to this one are \cite{hashe20,mckay18}.
Hasheminezhad and McKay~\cite{hashe20} obtained the asymptotic number of linear hypergraphs
with a given number of edges of each size, assuming
a constant bound on the edge size and $o(n^{\frac{4}{3}})$ edges.
McKay and Tian~\cite{mckay18}  obtained the asymptotic enumeration
formula for the set of $\mathcal{L}_r(n,m)$ as far as $m=o(n^{ \frac{3}{2}})$.
 Let $[x]_t=x(x-1)\cdots(x-t+1)$ be the falling factorial.
 % All asymptotics are with respect to $n\rightarrow\infty$.
The standard asymptotic notations $o$ and $O$ refer to $n\to\infty$.
The floor and ceiling signs are omitted whenever they are not crucial.

Let $s$ and $k=k(n)$ be integers with
$1\leqslant s\leqslant r\leqslant k\leqslant n$, and  $N_s$ be an abbreviation
 for $ \binom{n}{s}$. Let $A_1,\ldots,A_k$ be a given $k$-partition
of $[n]$ with $|A_i|=n_i\geqslant 1$, ${\bf \vec{n}}=(n_1,\ldots,n_k)$ and $\sigma_{s}({\bf \vec{n}})
=\sum_{1\leqslant i_1<\cdots<i_s\leqslant k}n_{i_1}\cdots n_{i_{s}}$
be the $s$-th elementary symmetric function of ${\bf \vec{n}}$. We use $A_{i_1}\cdots A_{i_s}$
to denote the set of $s$-sets $F_s$ of $[n]$ such that $|F_s \cap A_{i_j}|= 1$
for all $1\leqslant j\leqslant s$, and $\mathcal{E}_s({\bf \vec{n}})=\bigcup_{1\leqslant i_1<\cdots<i_s\leqslant k}
 A_{i_1}\cdots A_{i_s}$  for all $1\leqslant s\leqslant r$.
 An $r$-graph $H$ is called {\it $k$-partite} if each edge $e$ satisfies
$e\in \mathcal{E}_r({\bf \vec{n}})$. Let $\mathcal{H}_r({\bf \vec{n}},m)$
 denote the set of $k$-partite $r$-graphs with $m$  edges and with vertex partition
 determined by ${\bf \vec{n}}$, and let $\mathcal{L}_r({\bf \vec{n}},m)$ denote the set of all
 linear hypergraphs in $\mathcal{H}_r({\bf \vec{n}},m)$.
In this paper, we obtain an asymptotic enumeration formula for
$|\mathcal{L}_r({\bf \vec{n}},m)|$ as far as $m=o(n^{ \frac{4}{3}})$.

\begin{theorem}\label{t1.7}
 For a fixed integer $r\geqslant 3$, let $s$ and $k=k(n)$ be integers with
$1\leqslant s\leqslant r\leqslant k\leqslant n$,
and let $m=m(n)$ be an integer with
$m=o(n^{ \frac{4}{3}})$. Let ${\bf \vec{n}}=(n_1,\ldots,n_k)$ and
$\sigma_{s}({\bf \vec{n}})=\sum_{1\leqslant i_1<\cdots<i_s\leqslant k}n_{i_1}\cdots n_{i_{s}}$
be the $s$-th elementary symmetric function of ${\bf \vec{n}}$.
 Suppose that there exists a constant $C>0$ such that  $\sum_{i=1}^k \frac{1}{n_i}\leqslant C \frac{k^2}{n}$.
Then, as $n\rightarrow \infty$
\begin{align*}
|\mathcal{L}_r({\bf \vec{n}},m)|={ \frac{ \sigma_{r}^m({\bf \vec{n}})}{m!}}
\exp\biggl[- \frac{\sigma_{2}({\bf \vec{n}})
\sigma_{r-2}^2({\bf \vec{n}})[m]_2}{2\sigma_{r}^2({\bf \vec{n}})}+
O\biggl( \frac{m^2}{n^{3}}+ \frac{m^3}{n^{4}}\biggr)\biggr].
\end{align*}
\end{theorem}

Note that if there exists
a constant $c>0$ such that  $n_i\geqslant c \frac{n}{k}$ for $1\leqslant i\leqslant k$,
Theorem~\ref{t1.7}  holds. Also, for example, if $n_1=\ldots=n_p=n^{ \frac{1}{2}}$ for some
positive constant $p<k$, $n_{p+1}=\ldots=n_k=1$,
then $k=n+p-pn^{ \frac{1}{2}}$ and Theorem~\ref{t1.7}  holds; if
$n_1=\ldots=n_\ell=c'n$, $n_{\ell+1}=\ldots=n_k=1$
for some positive constants $\ell<k$ and $c'$ such that $c'\ell<1$, then
$k=(1- c'\ell) n+\ell$ and Theorem~\ref{t1.7}  holds.  For $n$ sufficiently large,
many cases satisfy
 $\sum_{i=1}^k \frac{1}{n_i}\leqslant C \frac{k^2}{n}$ for some
 constant $C>0$.
In particular,
for  $k=n$,  $k$-partite $r$-graphs are general $r$-graphs,
$\sigma_{2}({\bf \vec{n}})= N_2$, $\sigma_{r-2}({\bf \vec{n}})= N_{r-2}$
and $\sigma_{r}({\bf \vec{n}})= N_r$. We have the following corollary
on the  number of linear $r$-graphs on $[n]$ with $m=o(n^{ \frac{4}{3}})$
edges, which coincides with the uniform case in~\cite{hashe20} and is a subcase
in~\cite{mckay18}.

 \begin{corollary}\label{c11.2}
For any fixed integer $r\geqslant 3$, let  $m=m(n)$ be an integer
with $m=o(n^{ \frac{4}{3}})$. Then, as $n\rightarrow \infty$,
\begin{align*}
|\mathcal{L}_r(n,m)|
={ \frac{ N_r^m}{m!}}
\exp\biggl[- \frac{[r]_2^2[m]_2}{4n^2}+O\biggl( \frac{m^2}{n^{3}}+ \frac{m^3}{n^{4}}\biggr)\biggr].
\end{align*}
\end{corollary}

The remainder of the paper is structured as follows. Lemmas
are presented in Section~\ref{s:2}. In Section~\ref{s:3},
we complete the enumeration of $\mathcal{L}_r({\bf \vec{n}},m)$ with $m=o(n^{ \frac{4}{3}})$.

\section{Some Lemmas}\label{s:2}

In order to identify several events which have low probabilities
in the uniform probability space $\mathcal{H}_r({\bf \vec{n}},m)$ with
$m=o(n^{ \frac{4}{3}})$, the following lemmas  will be useful.

\begin{lemma}\label{l2.3}
For a fixed integer $r\geqslant 3$, let $s$ and $k=k(n)$ be integers with
$1\leqslant s\leqslant r\leqslant k\leqslant n$.
Let $\sigma_{s}({\bf \vec{n}})$ be the $s$-th elementary symmetric
function of ${\bf \vec{n}}=(n_1,\ldots,n_k)$.
% ${\bf \vec{n}}=(n_1,\red{\ldots},n_k)$ and
%$\sigma_{s}({\bf \vec{n}})=\sum_{1\leqslant i_1<\cdots<i_s\leqslant k}n_{i_1}\cdots n_{i_{s}}$
%be the $s$-th symmetric function of ${\bf \vec{n}}$.
 Suppose that there exists a constant $C>0$ such that  $\sum_{i=1}^k
 \frac{1}{n_i}\leqslant C \frac{k^2}{n}$.
%Suppose that there exists a constant $c>0$ such that $n_i\geqslant
%c \frac{n}{k}$ for $1\leqslant i\leqslant k$.
Then  $\sigma_{s}({\bf \vec{n}})= O( n^{s-r})\sigma_r({\bf \vec{n}})$.
\end{lemma}

\begin{proof} Let $S_j({\bf \vec{n}})=\binom{k}{j}^{-1}\sigma_j({\bf \vec{n}})$ for all $j=0,\ldots,k$. It is
clear that $S_{k-1}({\bf \vec{n}})=k^{-1}\sum_{i=1}^k \frac{n_1\cdots n_k}{n_i}$
and $S_{k}({\bf \vec{n}})=n_1\cdots n_k$. By Newton's inequality,
we have $S_{j-1}({\bf \vec{n}})S_{j+1}({\bf \vec{n}})\leqslant S_{j}^2({\bf \vec{n}})$,
and then
%By $n_i\geqslant c \frac{n}{k}$ for some constant $c>0$ and $1\leqslant i\leqslant k$,
%we have
\begin{equation*}\label{e2.3}
 \frac{S_{s-1}({\bf \vec{n}})}{S_{s}({\bf \vec{n}})}\leqslant \frac{S_{s}({\bf \vec{n}})}
{S_{s+1}({\bf \vec{n}})}\leqslant\cdots\leqslant \frac{S_{k-1}({\bf \vec{n}})}
{S_{k}({\bf \vec{n}})}= \frac{1}{k}\sum_{i=1}^k\frac{1}{n_i}\leqslant C \frac{k}{n}. %\frac{k}{cn}
\end{equation*}
Therefore
\[
    \frac{\sigma_s({\bf \vec{n}})}{\sigma_r({\bf \vec{n}})}
    = \frac{[r]_{r-s}}{[k-s]_{r-s}}\, \frac{S_s({\bf \vec{n}})}{S_r({\bf \vec{n}})}
    \leqslant \frac{[r]_{r-s}}{[k-s]_{r-s}} \, \frac{C^{r-s}k^{r-s}}{n^{r-s}}
    = O(n^{s-r}),
 \]
 where the last step holds since $s\leqslant r=O(1)$ and $k\geqslant r$ imply
 that $k^{r-s} = O([k-s]_{r-s})$.
%Then, $ \frac{\sigma_{r-1}({\bf \vec{n}})}{\sigma_{r}({\bf \vec{n}})}
%=  \frac{r S_{r-1}({\bf \vec{n}})}{(k-r+1)S_{r}({\bf \vec{n}}) }=O(n^{-1})$.}
%We apply the inequality $r-s$ times to complete the proof of Lemma~\ref{l2.3}.
\end{proof}

The following two lemmas are vector forms
of~\cite[Lemmas 2.1 and 2.2]{mckay18}. Their proofs are similar to those
in~\cite{mckay18}, but
Lemma~\ref{l2.3} is a key requirement in the proof of Lemma~\ref{l2.5}.

\begin{lemma}\label{l2.4}
For a fixed integer $r\geqslant 3$, let $k=k(n)$ be an integer with $r\leqslant k\leqslant n$,
and $H$ be chosen uniformly at random from $\mathcal{H}_r({\bf \vec{n}},m)$.
Let  $t=t(n)\geqslant 1$ be an integer and $e_1,\ldots,e_{t}$
be distinct $r$-sets in $\mathcal{E}_r({\bf \vec{n}})$.
Then the probability that $\{e_1,\ldots,e_t\}$ are  edges of $H$ is at most
$\bigl( \frac{m}{\sigma_{r}({\bf \vec{n}})}\bigr)^t$.
\end{lemma}

\begin{proof} Since $H$ is a $k$-partite $r$-graph
 that is chosen uniformly at random from $\mathcal{H}_r({\bf \vec{n}},m)$, the probability
that $e_1,\ldots,e_t$ are edges of $H$ is
\begin{align*}
 \frac{\binom{\sigma_{r}({\bf \vec{n}})-t}{m-t}}{\binom{\sigma_{r}({\bf \vec{n}})}{m}}
 &= \frac{[m]_t}{[\sigma_{r}({\bf \vec{n}})]_t}=
\prod_{i=0}^{t-1} \frac{m-i}{\sigma_{r}({\bf \vec{n}})-i}
\leqslant\Bigl( \frac{m}{\sigma_{r}({\bf \vec{n}})}\Bigr)^t.\qedhere
\end{align*}
\end{proof}

\begin{lemma}\label{l2.5}
Let $r\geqslant 3$, $t$ and $\alpha$ be integers such that $r,t,\alpha=O(1)$
and $0\leqslant\alpha\leqslant rt$. For any integer $k=k(n)$ with $r\leqslant k\leqslant n$,
let $H$ be chosen uniformly at random from $\mathcal{H}_r({\bf \vec{n}},m)$.
If there exists a constant $C>0$ such that  $\sum_{i=1}^k \frac{1}{n_i}\leqslant C \frac{k^2}{n}$,
%$n_i\geqslant c \frac{n}{k}$ for some constant $c>0$ and $1\leqslant i\leqslant k$,
then
 the expected number of sets of $t$ edges whose union has
$rt-\alpha$ or fewer vertices is $O(m^tn^{-\alpha})$.
\end{lemma}

\begin{proof} Let $e_1,\ldots,e_t$ be distinct $r$-sets in $\mathcal{E}_r({\bf \vec{n}})$.
We first bound the number of sequences $e_1,\ldots,e_t$ such that $|e_1\cup\cdots\cup e_t|=rt-\beta$
 for some $\beta$ with $ \alpha\leqslant \beta<rt$, regardless of whether they are edges of $H$.
For $2\leqslant i\leqslant t$, define $a_i=|(e_1\cup\cdots\cup e_{i-1})\cap e_i|$, thus we have
 $\sum_{i=2}^ta_i=\beta$. The first $r$-set $e_1$ can be chosen in $\sigma_{r}({\bf \vec{n}})$
 ways, then for $2\leqslant i\leqslant t$, the number of choices for $e_i$ given $e_1,\ldots,e_{i-1}$
 is at most $(rt)^{a_i}\sigma_{r-a_i}({\bf \vec{n}})$. Note that by Lemma~\ref{l2.3},
 $\sigma_{r-a_i}({\bf \vec{n}})=\sigma_{r}({\bf \vec{n}})O(n^{-a_i})$.
  Therefore, the number of choices of $e_1,\ldots,e_t$ for given $\beta,a_2,\ldots,a_t$ is at most
 $O(1)\sigma_{r}^t({\bf \vec{n}})\prod_{i=2}^{t}(rt)^{a_i}n^{-a_i}=O(\sigma_{r}^t({\bf \vec{n}})n^{-\beta})$.

The number of choices of $a_2,\ldots,a_t$ given $\beta$ is at most $O(1)$ by $r, t=O(1)$.
Also, by Lemma~\ref{l2.4}, the probability that $e_1,\ldots,e_t\in H$ is at most
$(m/\sigma_{r}({\bf \vec{n}}))^t$. Therefore, the expected number of sets of $t$ edges
of $H$ whose union has size $rt-\beta$ is $O(m^tn^{-\beta})$, uniformly over $\beta$.
Finally, the sum of this expression over $\beta\geqslant\alpha$ is bounded by a decreasing geometric series
dominated by the term $\beta=\alpha$. This completes the proof.
\end{proof}

We also need the following Lemma from~\cite{green06}, which was used to enumerate
some hypergraphs in~\cite{vlaejoc,valelec,hashe20,mckay18}.

\begin{lemma}[\cite{green06}, Corollary~4.5]\label{l2.6}
Let $N\geqslant 2$ be an integer, and for $1\leqslant i\leqslant N$, let
real numbers $A(i)$, $B(i)$ be given such that
$A(i)\geqslant 0$ and $1-(i-1)B(i)\geqslant 0$. Define $A_1=\min_{i=1}^NA(i)$,
$A_2=\max_{i=1}^NA(i)$, $C_1=\min_{i=1}^NA(i)B(i)$
and $C_2=\max_{i=1}^NA(i)B(i)$. Suppose that there exists
a real number $\hat{c}$ with $0<\hat{c}< \frac{1}{3}$
such that $\max\{A/N,|C|\}\leqslant \hat{c}$ for all $A\in [A_1,A_2]$, $C\in[C_1,C_2]$.
Define $h_0$, $h_1$, $\ldots$, $h_N$ by $h_0=1$ and $ \frac{h_i}{h_{i-1}}= \frac{A(i)}{i}(1-(i-1)B(i))$
for $1\leqslant i\leqslant N$, with the following interpretation:
if $A(i)= 0$ or $1-(i-1)B(i)=0$, then $h_j=0$
for $i\leqslant j\leqslant N$. Then $\Sigma_1\leqslant \sum_{i=0}^{N}h_i\leqslant \Sigma_2$,
where $\Sigma_1=\exp[A_1- \frac{1}{2}A_1C_2]-(2e\hat{c})^N$ and
$\Sigma_2=\exp[A_2- \frac{1}{2}A_2C_1+ \frac{1}{2}A_2C_1^2]+(2e\hat{c})^N$.
\end{lemma}

\section{Enumeration of $\mathcal{L}_r({\bf \vec{n}},m)$
with $m=o(n^{ \frac{4}{3}})$}\label{s:3}

Let~$H$ be a~$k$-partite $r$-graph in $\mathcal{H}_r({\bf \vec{n}},m)$.
As defined in~\cite{mckay18}, a~$2$-set $\{x,y\}\subseteq [n]$
is called a \textit{link} if there are two distinct edges $e,f$ such that
$\{x,y\}\subseteq e\cap f$.
The two edges
$e$ and $f$ are called \textit{linked edges} if $|e\cap f|\geqslant 2$.
Let $G_H$ be the simple graph whose vertices are the edges of~$H$,
with two vertices of $G$ adjacent iff the corresponding edges of~$H$ are linked.
An edge-induced subgraph of $H$ corresponding to a non-trivial
component of $G_H$ is called a \textit{cluster} of $H$.

Let $\mathbb{P}_r({\bf \vec{n}},m)$ denote the probability that a~$k$-partite
$r$-graph $H\in \mathcal{H}_r({\bf \vec{n}},m)$ chosen uniformly at random is linear.
Hence,
\begin{align}\label{e3.1}
|\mathcal{L}_r({\bf \vec{n}},m)|=\binom{\sigma_{r}({\bf \vec{n}})}{m} \mathbb{P}_r({\bf \vec{n}},m).
\end{align}
We will prove that $\mathbb{P}_r({\bf \vec{n}},m)$ equals the exponential factor
in Theorem~\ref{t1.7}.

Firstly, we show that most of $\mathcal{H}_r({\bf \vec{n}},m)$
has a simple structure.
Define $\mathcal{H}_r^+({\bf \vec{n}},m)\subseteq\mathcal{H}_r({\bf \vec{n}},m)$
to be the set of $k$-partite $r$-graphs $H$ which satisfy the following
two properties $\bf(a)$ and $\bf(b)$.

%%%%NEW%%%%%%%%%%%%
$\bf(a)$\   Every cluster of $H$ consists of two edges overlapping by two vertices.

$\bf(b)$\  The number of clusters in $H$ is at most $M$,
where  $M=\Bigl\lceil\log n+ \frac{56\sigma_{r-2}^2({\bf \vec{n}})\sigma_2({\bf \vec{n}})m^2}
{\sigma_{r}^2({\bf \vec{n}})}\Bigr\rceil$.

We show that  the  expected number of $k$-partite $r$-graphs in $\mathcal{H}_r({\bf \vec{n}},m)$
not satisfying the properties of $\mathcal{H}_r^+({\bf \vec{n}},m)$ is quite small.

\begin{lemma}\label{l6.1}
For a fixed integer $r\geqslant 3$, let $k=k(n)$ and $m=m(n)$ be integers with
$r\leqslant k\leqslant n$ and $m=o(n^{ \frac{4}{3}})$. Then, as $n\rightarrow \infty$,
$ \frac{|\mathcal{H}_r^+({\bf \vec{n}},m)|}{|\mathcal{H}_r({\bf \vec{n}},m)|}
 =1-O\bigl( \frac{m^2}{n^3}+ \frac{m^3}{n^4}\bigr)$.
\end{lemma}

\begin{proof} Consider $H\in \mathcal{H}_r({\bf \vec{n}},m)$ chosen uniformly at random.
 We apply Lemma~\ref{l2.5} several times to show that $H$ satisfies
 the properties $\bf(a)$ and $\bf(b)$ with probability $1-O\( \frac{m^2}{n^3}+ \frac{m^3}{n^4}\)$.

If two edges overlap by three or more vertices, then they have at most $2r-3$ vertices in
total, which has probability $O\( \frac{m^2}{n^{3}}\)$ by Lemma~\ref{l2.5}.
Similarly if there is a cluster of more than two edges, then three of those edges have
at most $3r-4$ vertices in total, which has probability
$O\( \frac{m^3}{n^{4}}\)$ by Lemma~\ref{l2.5}.
%Applying Lemma~\ref{l2.5} with $t=2$ and $\alpha=3$, the expected number of
%two edges involving at most $2r-3$ vertices is $O( \frac{m^2}{n^{3}})$.
%with $t=3$ and $\alpha=4$, the expected number of three edges involving at most $3r-4$
%vertices is $O( \frac{m^3}{n^{4}})$.
%\red{Thus, if there is a cluster with two edges $e_1$ and
%$e_2$ of $H$ such that $|e_1\cap e_2|=2$, then $|e\cap (e_1\cup e_2)|\leqslant 1$  for any
%edge $e$ other than $e_1,e_2$ of $H$,
Therefore, $H$ satisfies the property $\bf(a)$ %, $\bf(b)$ and $\bf(c)$
with probability $1-O\( \frac{m^2}{n^3}+ \frac{m^3}{n^4}\)$.

%\bdm{Having $2d$ distinct edges does not imply that the links are distinct.
%Since there are no big clusters, this part of the proof can be written more simply.}--$\bf(c)$
Note that if $\bf(a)$ holds, all clusters have two edges and
no two clusters share an edge or a link.
Define the event
\[  \mathcal{D}=\{\text{there exist at least $d$ edge- and link-disjoint clusters in $H$}\}, \]
where $d=M+1$.
Using Lemma~\ref{l2.4}, we have
\begin{align*}
\mathbb{P}[\mathcal{D}]&=O\biggl(\sigma_{r-2}^{2d}({\bf \vec{n}})
{\sigma_2({\bf \vec{n}})\choose d}\biggl(\frac{m}{\sigma_{r}({\bf \vec{n}})}\biggr)^{\!2d\,}\biggr)\\
&=O\biggl(\biggl( \frac{e\sigma_2({\bf \vec{n}})\sigma_{r-2}^2({\bf \vec{n}})m^2}
{d \sigma_{r}^2({\bf \vec{n}})}\biggr)^{\!d\,}\biggr)\\
&=O\Bigl( \Bigl(\frac{e}{56}\Bigr)^{\!d\,}\Bigr)\\
&=O\Bigl( \frac{1}{n^{3}}\Bigr),
\end{align*}
where the last two inequalities are true because
$d> \frac{56\sigma_{r-2}^2({\bf \vec{n}})\sigma_2({\bf \vec{n}})m^2}
{\sigma_{r}^2({\bf \vec{n}})}$ and  $d>\log n$.
The proof is complete on noting that the event ``$\bf(a)$ and $\bf(b)$ hold''
is contained in the union of the events ``$\bf(a)$ holds''
and ``$\mathcal{D}$ doesn't hold''.
\end{proof}

From the proof of Lemma~\ref{l6.1}, we have $|\mathcal{H}_r^+({\bf \vec{n}},m)|\neq0$.
Hence, there exists a nonnegative integer $t$ such that the set of $k$-partite
$r$-graphs with exactly $t$ clusters in $\mathcal{H}_r^+({\bf \vec{n}},m)$
is nonempty and is denoted by $\mathcal{C}_{t}^{+}$. By the definition of
$\mathcal{H}_r^+({\bf \vec{n}},m)$ we have $|\mathcal{H}_r^+({\bf \vec{n}},m)|
=\sum_{t=0}^{M}|\mathcal{C}_{t}^{+}|$. By the switching operations below,
we will prove that $\mathcal{L}_r({\bf \vec{n}},m)
=\mathcal{C}_{0}^{+}\neq\emptyset$.
It follows that
\begin{align}\label{e3.11}
 \frac{1}{\mathbb{P}_r({\bf \vec{n}},m)}&=\Bigl(1-O\Bigl( \frac{m^2}{n^3}+
 \frac{m^3}{n^4}\Bigr)\Bigr)\sum_{t=0}^{M}
 \frac{|\mathcal{C}_{t}^{+}|}{|\mathcal{L}_r({\bf \vec{n}},m)|}
 =\Bigl(1-O\Bigl( \frac{m^2}{n^3}+ \frac{m^3}{n^4}\Bigr)\Bigr)\sum_{t=0}^{M}
\frac{|\mathcal{C}_{t}^{+}|}{|\mathcal{C}_{0}^{+}|}.
\end{align}

In order to find the ratio $ |\mathcal{C}_{t}^{+}|/|\mathcal{C}_{0}^{+}|$
when $1\leqslant t\leqslant M$, we design switchings to find a relationship between the sizes of
 $\mathcal{C}_{t}^{+}$ and $\mathcal{C}_{t-1}^{+}$. Let $H\in \mathcal{C}_{t}^{+}$.
  A {\it forward switching} from $H$ is used to reduce the number of clusters
  in $H$. Take any cluster consisting of two edges $e$ and $f$, and remove
  them from $H$. Define $H_0$ with the the same vertex set $[n]$ and
the edge set $E(H_0)=E(H)\setminus \{e,f\}$. Choose any $r$-set $e_1$ from
%does not share a link with any edge of $H_0$
$ \mathcal{E}_r({\bf \vec{n}})$ such that $e_1$ does not share a link with any edge of $H_0$,
%has at most one vertex in common
%with any edge of $E(H_0)$and further at most one vertex in common with each cluster in $H_0$
and define $H'$ by setting $E(H')=E(H_0)\cup \{e_1\}$.
Next, similarly choose another $r$-set $e_2$ from $ \mathcal{E}_r({\bf \vec{n}})$ such
that $e_2$ does not share a link with any edge of $H'$.
%does not share a link with any edge of $H'$.has at most one vertex in common
%with any edge of $E(H')$ and further at most one vertex in common with each cluster in $H'$
Add edge $e_2$ to $H'$
to produce $H''$, which is the result of the forward switching from $H$.
Note that the two edges $e_1$ and $e_2$ may have at most one vertex in common
and $H''\in \mathcal{C}_{t-1}^{+}$.

A {\it reverse switching} is the reverse of a  forward switching.
Let $H''\in \mathcal{C}_{t-1}^{+}$. Sequentially choose two edges $e_1$
and $e_2$ of $H''$ such that neither of them contains a link.
Define $H_0$ with the same vertex set $[n]$ and
$E(H_0)=E(H'')\setminus \{e_1,e_2\}$. Take two $r$-sets $e$ and $f$ in $ \mathcal{E}_r({\bf \vec{n}})$
such that $|e\cap f|=2$ and neither of them share a link with any edge of $H_0$.
Insert $e$ and $f$ into $H_0$. Call the resulting graph $H$. Then, $H\in \mathcal{C}_{t}^{+}$.

\begin{lemma}\label{l4.1}
For any fixed integer $r\geqslant 3$, let $k=k(n)$ and $m=m(n)$ be integers
with $r\leqslant k\leqslant n$ and $m=o(n^{ \frac{4}{3}})$. Let $t$ be some positive integer
with $1\leqslant t\leqslant M$.\\
$(a)$\ Let $H\in \mathcal{C}_{t}^{+}$. The number of  forward switchings for $H$ is
$t\sigma_r^2({\bf \vec{n}})\(1+O\( \frac{m}{n^{2}}\)\)$.

\noindent$(b)$\ Let $H''\in \mathcal{C}_{t-1}^{+}$. The number of
reverse switchings for $H''$ is
${m-2(t-1)\choose 2}\sigma_2({\bf \vec{n}})\sigma_{r-2}^2({\bf \vec{n}})
\(1+O\( \frac{1}{n}+ \frac{m}{n^{2}}\)\)$.
\end{lemma}

\begin{proof} $(a)$\  Let $H\in \mathcal{C}_{t}^{+}$. Let $\mathcal{R}(H)$ be the set of all
forward switchings which can be applied to~$H$. There are exactly $t$ ways to choose a
cluster; remove it from $H$ to produce $H_0$. The number of choices for the $r$-set
$e_1$ is at most $\sigma_r({\bf \vec{n}})$. From this we must  subtract
the number of $r$-sets that overlap some  edge of $H_0$ in two or more vertices,
which is at most $ \binom{r}{2}(m-2)\sigma_{r-2}({\bf \vec{n}})=
O( \frac{m}{n^{2}})\sigma_r({\bf \vec{n}})$  by Lemma~\ref{l2.3} and $r=O(1)$.
%\red{We further subtract the number of $r$-sets that overlap some cluster of $H_0$ in two or more vertices,
%which is at most $ (t-1)r^2\sigma_{r-2}({\bf \vec{n}})=
%O( \frac{m}{n^{2}})\sigma_r({\bf \vec{n}})$  by Lemma~\ref{l2.3}, $2t\leqslant m$ and $r=O(1)$.}
Thus, there are $\sigma_r({\bf \vec{n}})(1+O( \frac{m}{n^{2}}))$ ways to
choose $e_1$. Similarly, there are $\sigma_r({\bf \vec{n}})(1+O( \frac{m}{n^{2}}))$
ways to choose $e_2$. We have $|\mathcal{R}(H)|=t\sigma_r^2({\bf \vec{n}})(1+O( \frac{m}{n^{2}}))$.

$(b)$\ Conversely, suppose that $H''\in \mathcal{C}_{t-1}^{+}$.
Similarly, let $\mathcal{R}'(H'')$ be the set of all reverse switchings for $H''$.
There are exactly $2\binom{m-2(t-1)}{2}$ ways to delete two edges in sequence such that
neither of them contains a link in $H''$. Let the resulting graph be $H_0$.
There are at most $ \frac{1}{2}\sigma_2({\bf \vec{n}})\sigma_{r-2}^2({\bf \vec{n}})$
ways to choose two $r$-sets $e$ and $f$ in $ \mathcal{E}_r({\bf \vec{n}})$
such that $e$ and $f$ are linked edges. From this we firstly subtract the ones
with $|e\cap f|\geqslant 3$, which is at most
$\frac{r}{2}\sigma_2({\bf \vec{n}})\sigma_{r-2}
({\bf \vec{n}})\sigma_{r-3}({\bf \vec{n}})= \sigma_2({\bf \vec{n}})\sigma_{r-2}^2({\bf \vec{n}})
O( \frac{1}{n})$ by Lemma~\ref{l2.3}, since $r=O(1)$. Secondly, we subtract the
cases where at least one of $e,f$ (say $e$) shares a link %has two or more vertices in common
with one of the $m-2$ edges $e'$ in $H_0$.
Let $\ell_1$ be the link shared by $e,f$ and $\ell_2$ be the link shared by $e,e'$.
The number of cases for $|\ell_1\cup\ell_2|=2,3,4$ is at most
$mr^2\sigma_{r-2}^2({\bf \vec{n}})$,
$mr^2 \sigma_1({\bf \vec{n}})\sigma_{r-3}({\bf \vec{n}})\sigma_{r-2}({\bf \vec{n}})$
and $mr^2\sigma_2({\bf \vec{n}})\sigma_{r-4}({\bf \vec{n}})\sigma_{r-2}({\bf \vec{n}})$,
respectively, where the last one is only possible if $r\geqslant 4$. By Lemma~\ref{l2.3}, each
of these expressions is $\sigma_2({\bf \vec{n}})\sigma_{r-2}^2({\bf \vec{n}})O( \frac{m}{n^{2}})$.
This completes the proof.
%\red{Finally, we further subtract the case when the two $r$-sets $e$ and $f$ with $|e\cap f|=2$
%have two or more vertices in common with the cluster of $H_0$, which is at most
%$(t-1)r^2\sigma_2({\bf \vec{n}})\sigma_{r-3}^2({\bf \vec{n}})=
%\sigma_2({\bf \vec{n}})\sigma_{r-2}^2({\bf \vec{n}})O( \frac{m}{n^{2}})$
%by Lemma~\ref{l2.3}, $2t\leqslant m$ and $r=O(1)$.}
% which is at most
%$ \frac{1}{2}\sigma_2({\bf \vec{n}})\sigma_{r-2}^2({\bf \vec{n}})O( \frac{m}{n^{2}})$ by Lemma~\ref{l2.3}.
%Thus, we have $ |\mathcal{R}'(H'')|= \frac{1}{2}\sigma_2({\bf \vec{n}})\sigma_{r-2}^2({\bf \vec{n}})
%(1+O( \frac{1}{n}+ \frac{m}{n^{2}}))$.
\end{proof}

\begin{corollary}\label{c4.2}
With notation as above, for some $1\leqslant t\leqslant M$,\\
$(a)$\ $|\mathcal{C}_{t}^{+}|>0$ iff $m\geqslant 2t$.\\
$(b)$\ Let $t'$ be the first value of $t\leqslant M$ such that $\mathcal{C}_{t}^{+}=\emptyset$,
or $t'=M+1$ if no such value exists. Then, as $n\rightarrow\infty$,
uniformly for $1\leqslant t< t'$,
\begin{align*}
 \frac{|\mathcal{C}_{t}^{+}|}{|\mathcal{C}_{t-1}^{+}|}&=
 \frac{ \binom{m-2(t-1)}{2}\sigma_2({\bf \vec{n}})\sigma_{r-2}^2({\bf \vec{n}})}
 {t\sigma_r^2({\bf \vec{n}})}\Bigl(1+O\Bigl( \frac{1}{n}+ \frac{m}{n^{2}}\Bigr)\Bigr).
 \end{align*}
\end{corollary}

\begin{proof} $(a)$\ Firstly, $m\geqslant 2t$ is necessary for $|\mathcal{C}_{t}^{+}|>0$.
By Lemma~\ref{l6.1}, there is some $0\leqslant \hat{t}\leqslant M$ such that
$\mathcal{C}_{\hat{t}}^{+}\neq\emptyset$. We can move $\hat{t}$ to $t$ by a sequence of forward
and reverse switchings while no greater than $M$. Note that since the values
given in Lemma~\ref{l4.1} at each step of this path are positive for any
$0\leqslant t\leqslant M$, we have $|\mathcal{C}_{t}^{+}|>0$.

$(b)$\ By $(a)$,  if $\mathcal{C}_{t}^{+}=\emptyset$, then
$\mathcal{C}_{t+1}^{+},\ldots, \mathcal{C}_{M}^{+} =\emptyset$.
By the definition of $t'$, if $1\leqslant t<t'$, then the left hand ratio is well defined.
By Lemma~\ref{l4.1} completes the proof.
\end{proof}

At last, we estimate the sum $\sum_{t=0}^{M}
 {|\mathcal{C}_{t}^{+}|}/{|\mathcal{C}_0^+|}$
by applying Lemma~\ref{l2.6}, which is used to count
certain hypergraphs in~\cite{vlaejoc,valelec,hashe20,mckay18}.

\begin{lemma}\label{l4.3}
For any given integer $r\geqslant 3$, let $k=k(n)$ and $m=m(n)$ be integers
with $r\leqslant k\leqslant n$ and $m=o(n^{ \frac{4}{3}})$. With notation above,
as $n\rightarrow\infty$,
\begin{align*}
\sum_{t=0}^{M} \frac{|\mathcal{C}_{t}^{+}|}{|\mathcal{C}_{0}^{+}|}
&=\exp\biggl[ \frac{\sigma_2({\bf \vec{n}})\sigma_{r-2}^2({\bf \vec{n}})[m]_2}
{2\sigma_r^2({\bf \vec{n}})}+O\Bigl( \frac{m^2}{n^3}+ \frac{m^3}{n^4}\Bigr)\biggr].
\end{align*}
\end{lemma}

\begin{proof} Let $t'$ be as  defined in Corollary~\ref{c4.2}\,$(b)$.
We know that  $|\mathcal{C}_{0}^{+}|\neq0$, then $t'\geqslant 1$. If
$t'=1$, then we have $m<2$ from Corollary~\ref{c4.2}\,$(a)$,  and
the conclusion is obviously true. In the following, suppose $t'\geqslant 2$.
Define $h_{0},\ldots,h_{M}$ by $h_{0}=1$, $h_{t}= |\mathcal{C}_{t}^{+}|/|\mathcal{C}_{0}^{+}|$
for $1\leqslant t<t'$ and $h_{t}=0$ for $t'\leqslant t\leqslant M$.
By Corollary~\ref{c4.2}\,$(b)$, we have for $1\leqslant t< t'$,
\begin{equation}\label{e6.3}
\begin{aligned}[b]
 \frac{h_{t}}{h_{t-1}}&= \frac{1}{t}
 \binom{m-2(t-1)}{2} \frac{\sigma_2({\bf \vec{n}})\sigma_{r-2}^2({\bf \vec{n}})}{\sigma_r^2({\bf \vec{n}})}
\Bigl(1+O\Bigl( \frac{1}{n}+ \frac{m}{n^2}\Bigr)\Bigr).
\end{aligned}
\end{equation}
For $1\leqslant t\leqslant M$, define
\begin{equation}\label{e6.2}
\begin{aligned}[b]
A(t)&=  \frac{\sigma_2({\bf \vec{n}})\sigma_{r-2}^2({\bf \vec{n}})[m]_2}{2\sigma_r^2({\bf \vec{n}})}
+O\Bigl( \frac{m^2}{n^3}+ \frac{m^3}{n^4}\Bigr),\\
B(t)&=\begin{cases} \frac{2(2m-2t+1)}{m(m-1)},\text{for}\ 1\leqslant t<t';\\
(t-1)^{-1},\text{otherwise}.\end{cases}
\end{aligned}
\end{equation}
Using the equations shown in~\eqref{e6.3} and~\eqref{e6.2}, for $1\leqslant t< t'$, we further have
\begin{align*}
 \frac{h_{t}}{h_{t-1}}&=
\frac{A(t)}{t}\bigl(1-(t-1)B(t)\bigr).
\end{align*}
Following the notation of Lemma~\ref{l2.6}, we also have
\begin{align}\label{e6.6}
A_1,A_2= \frac{\sigma_2({\bf \vec{n}})\sigma_{r-2}^2({\bf \vec{n}})[m]_2}{2\sigma_r^2({\bf \vec{n}})}
+O\Bigl( \frac{m^2}{n^3}+ \frac{m^3}{n^4}\Bigr).
\end{align} For $1\leqslant t< t'$, we have
\begin{align*}
A(t)B(t)&= \frac{\sigma_2({\bf \vec{n}})\sigma_{r-2}^2({\bf \vec{n}})
(2m-2t+1)}{\sigma_r^2({\bf \vec{n}})}
+O\Bigl( \frac{m}{n^3}+ \frac{m^2}{n^4}\Bigr).
\end{align*}Then $A(t)B(t)= O\bigl( \frac{m}{n^2}\bigr)$
because  ${\sigma_{r-2}^2({\bf \vec{n}})}/
{\sigma_r^2({\bf \vec{n}})}=O( {n^{-4}})$ by applying Lemma~\ref{l2.3},
and $\sigma_2({\bf \vec{n}})=O(n^2)$ based on the fact that $\sigma_2({\bf \vec{n}})$
is the number of edges of complete $k$-partite graphs and its maximum value
occurs at the approximately equal partition. For the case $t'\leqslant t\leqslant M$
and $t'\geqslant 2$, by Corollary~\ref{c4.2}\,$(a)$, we have $2\leqslant m<2t$,
and then $A(t)B(t)= O\bigl(\frac{m}{n^2}\bigr)$ for $t'\leqslant t\leqslant M$.
In both cases, following the notation of Lemma~\ref{l2.6}, we have
\begin{align}\label{e6.4}
C_1,C_2=O\Bigl(\frac{m}{n^2}\Bigr).
\end{align}
Then $|C|=o(1)$ for all $C\in[C_{1},C_{2}]$ when $m=o(n^{ \frac{4}{3}})$.

Let $\hat{c}=\frac{1}{110}$. Note that $M=\bigl\lceil\log n+
\frac{56\sigma_{r-2}^2({\bf \vec{n}})\sigma_2({\bf \vec{n}})m^2}
{\sigma_{r}^2({\bf \vec{n}})}\bigr\rceil$. We have $ \frac{A}{M}\leqslant \frac{1+ O( 1/n+ m/n^2)}{112}$
for all $A\in [A_1,A_2]$ as shown in~\eqref{e6.6}. Thus, $\max\{ \frac{A}{M},|C|\}< \hat{c}$ and
$(2e\hat{c})^{M}=O( \frac{1}{n^3})$ as $n\rightarrow\infty$. Using the equations
shown in~\eqref{e6.6} and~\eqref{e6.4}, we have
\begin{align*}
A_1C_2,A_2C_1=O\Bigl( \frac{m^2}{n^2}\cdot \frac{m}{n^2}\Bigr)=
O\Bigl( \frac{m^3}{n^4}\Bigr) \ \text{and}\ A_2C_1^2=
O\Bigl( \frac{m^2}{n^2}\cdot \frac{m^2}{n^4}\Bigr)=
O\Bigl( \frac{m^4}{n^6}\Bigr).
\end{align*}
Lemma~\ref{l2.6} applies to obtain
\begin{align*}
\sum_{t=0}^{M} \frac{|\mathcal{C}_{t}^{+}|}{|\mathcal{C}_{0}^{+}|}
&=\exp\Bigl[ \frac{\sigma_2({\bf \vec{n}})
\sigma_{r-2}^2({\bf \vec{n}})[m]_2}{2\sigma_r^2({\bf \vec{n}})}+
O\Bigl( \frac{m^{2}}{n^3}+ \frac{m^{3}}{n^4}\Bigr)\Bigr]+ O\Bigl( \frac{1}{n^3}\Bigr)\\
&=\exp\Bigl[ \frac{\sigma_2({\bf \vec{n}})
\sigma_{r-2}^2({\bf \vec{n}})[m]_2}{2\sigma_r^2({\bf \vec{n}})}+
O\Bigl( \frac{m^{2}}{n^3}+ \frac{m^{3}}{n^4}\Bigr)\Bigr]
\end{align*}when $m=o(n^{ \frac{4}{3}})$.
\end{proof}

\begin{proof} [Proof of Theorem~\ref{t1.7} ]By applying Lemma~\ref{l4.3},
using the equations shown in~\eqref{e3.1} and~\eqref{e3.11}, we have
\begin{align*}
|\mathcal{L}_r({\bf \vec{n}},m)|&=\binom{\sigma_{r}({\bf \vec{n}})}{m}
\exp\Bigl[ -\frac{\sigma_2({\bf \vec{n}})\sigma_{r-2}^2({\bf \vec{n}})[m]_2}{2\sigma_r^2({\bf \vec{n}})}+
O\Bigl( \frac{m^{2}}{n^3}+ \frac{m^{3}}{n^4}\Bigr)\Bigr]\\
&= \frac{\sigma_{r}^m({\bf \vec{n}})}{m!} \exp\Bigl[ -\frac{\sigma_2({\bf \vec{n}})
\sigma_{r-2}^2({\bf \vec{n}})[m]_2}{2\sigma_r^2({\bf \vec{n}})}+
O\Bigl( \frac{m^{2}}{n^3}+ \frac{m^{3}}{n^4}\Bigr)\Bigr],
\end{align*}
since
\begin{align*}
\binom{\sigma_{r}({\bf \vec{n}})}{m}= \frac{\sigma_{r}^m({\bf \vec{n}})}{m!}
\exp\Bigl[O\Bigl( \frac{m^{2}}{\sigma_{r}({\bf \vec{n}})}\Bigr)\Bigr]
= \frac{\sigma_{r}^m({\bf \vec{n}})}{m!}\exp\Bigl[O\Bigl( \frac{m^{2}}{n^3}\Bigr)\Bigr]
\end{align*}
because $\sigma_{r}({\bf \vec{n}})\geqslant \binom{k}{r}c^r ( \frac{n}{k})^r$ and $r\geqslant 3$.
\end{proof}
\begin{proof}[Proof of Corollary~\ref{c11.2}]
If  $k=n$, then $\sigma_{2}({\bf \vec{n}})= N_2$, $\sigma_{r-2}({\bf \vec{n}})= N_{r-2}$,
$\sigma_{r}({\bf \vec{n}})= N_r$ and $k$-partite $r$-graphs are $r$-graphs.
By applying Theorem~\ref{t1.7}, it follows that
\begin{align*}
|\mathcal{L}_r(n,m)|
={ \frac{ N_r^m}{m!}}
\exp\Bigl[- \frac{[r]_2^2[m]_2}{4n^2}+O\Bigl( \frac{m^2}{n^{3}}+ \frac{m^3}{n^{4}}\Bigr)\Bigr]
\end{align*}
when $m=o(n^{ \frac{4}{3}})$.
\end{proof}

\begin{remark}
The formula in Corollary~\ref{c11.2} coincides with the uniform case in~\cite{hashe20}
and is a subcase in~\cite{mckay18}. Compared with the enumeration formula of
$|\mathcal{L}_r(n,m)|$ in~\cite{mckay18},
\begin{align*}
|\mathcal{L}_r(n,m)|
&={ \frac{ N_r^m}{m!}}\exp\biggl[- \frac{[r]_2^2[m]_2}{4n^2}-
\frac{[r]_2^3(3r^2-15r+20)m^3}{24n^4}+O\Bigl( \frac{m^2}{n^3}\Bigr)\biggr],
\end{align*}
 the
 term $\frac{[r]_2^3(3r^2-15r+20)m^3}{24n^4}=O\bigl(\frac{m^3}{n^4}\bigr)$ is under
 our conditions.
It will take new ideas to handle larger $m$ in
$\mathcal{L}_r({\bf \vec{n}},m)$ and they will be more complicated
than those in~\cite{hashe20,mckay18}.
We leave these problems for future work.
\end{remark}

\section*{Acknowledgement}

Most  of this work was finished  when Fang Tian was a visiting research fellow at
Australian National University. She is very grateful for what she learned there.
She is also immensely grateful to the anonymous reviewer for his/her detailed and
helpful suggestions.
%Without his/her help, it is impossible for her to write
%this manuscript in its current form.

\end{document}